\documentclass{amsart}
\usepackage{graphicx}
\newtheorem{thm}{Theorem}[section]
\newtheorem{cor}[thm]{Corollary}
\newtheorem{lem}[thm]{Lemma}
\newtheorem{prop}[thm]{Proposition}
\theoremstyle{definition}
\newtheorem{defn}[thm]{Definition}
\newtheorem{eg}[thm]{Example}

\newtheorem{rem}[thm]{Remark}
\numberwithin{equation}{section}

\newcommand{\Spec}{\mathrm{Spec}}
\newcommand{\m}{\frak{m}}
\newcommand{\p}{\frak{p}}
\newcommand{\q}{\frak{q}}
\newcommand{\Ann}{\mathrm{Ann}}
\newcommand{\Max}{\mathrm{Max}}
\newcommand{\rad}{\mathrm{rad}}
\newcommand{\Rad}{\mathrm{Rad}}
\begin{document}

\title[Quasi-prime submodules and ...]{Quasi-prime submodules and developed Zariski topology}%
\author{A. Abbasi}%

\address{Department of Mathematics, Faculty of Science, University of Guilan, P. O. Box 41335-1914, Rasht, Iran.}%
\email{aabbasi@guilan.ac.ir}%
\author{D. Hassanzadeh-lelekaami}%
\address{Department of Mathematics, Faculty of Science, University of Guilan, P. O. Box 41335-1914, Rasht, Iran.}%
\email{dhmath@guilan.ac.ir}%

\date{\today}%
\begin{abstract}
Let $R$ be a commutative ring with nonzero identity and $M$ be an $R$-module. Quasi-prime submodules of $M$ and  the developed  Zariski topology on $q\Spec(M)$ are introduced. We also, investigate the relationship
between the algebraic properties of $M$ and the topological properties
of $q\Spec(M)$. Modules whose developed Zariski topology is respectively $T_0$,
irreducible or Noetherian are studied, and several characterizations of such modules
are given.
\end{abstract}
\maketitle

\section{INTRODUCTION}

Prime submodules of modules were introduced as a generalization of prime ideals of rings by J. Dauns \cite{Dau78} and several algebraists carried out an intensive and systematic study of the
spectrum of prime submodules (e.g. \cite{lu84}, \cite{mm92}, \cite{lu95}, \cite{mms97}, \cite{mms98}, \cite{lu99}, \cite{mms02}, \cite{lu07}). Here,  quasi-prime submodules of $M$ as a generalization of prime submodules are introduced.
We also, investigate the quasi-primeful modules and we apply them to
develop of topological properties of $q\Spec(M)$, where
$q\Spec(M)$ is the set of all quasi-prime submodules of $M$.

The Zariski topology on the spectrum of prime ideals of a ring is one of the main tools in Algebraic Geometry.
In the literature, there are many different
generalizations of the Zariski topology of rings to  modules (see \cite{mms97}, \cite{behI08},  \cite{behII08}, or \cite{lu99}).
In this paper, we are going study the developed Zariski topology as a
  generalization of the Zariski topology considered in \cite{lu99}, to $q\Spec(M)$, where $M$ is an $R$-module. As is well known, the Zariski topology has been defined on the set of all prime submodules of a module. Here, we considered  developed Zariski topology on the set  of all quasi-prime submodules of a module.

Throughout this paper, all rings are commutative with identity and all modules are unital. For a submodule $N$ of an $R$-module $M$, $(N :_R M)$ denotes the ideal $\{r\in R
\mid rM\subseteq N\}$ and annihilator of $M$, denoted by $\Ann_R(M)$,
is the ideal $(\textbf{0}:_R M)$. $M$ is called faithful if $\Ann(M)=(0)$. If there is no ambiguity we write $(N : M)$ (resp. $\Ann(M)$) instead of $(N :_R M)$ (resp. $\Ann_R(M)$).
A proper ideal $I$ of a ring $R$ is said to be \emph{quasi-prime} if for each pair of ideals $A$ and $B$ of $R$, $A \cap B \subseteq I$
yields either $A\subseteq I$ or $B \subseteq I$ (see \cite{az08}, \cite{Bou70} and \cite{hen02}). It is easy to see that every
 prime ideal is a quasi-prime ideal. Also, every quasi-prime ideal is irreducible (an ideal $I$ of a commutative ring $R$ is said to be \emph{irreducible} if $I$ is not the
intersection of two ideals of $R$ that properly contain it).

A submodule $N$ of an $R$-module $M$ is said to be \emph{prime} if $N\neq M$ and
whenever $rm \in N$ (where $r\in R$ and $m \in M$), then $r\in (N
: M)$ or $m \in N$. If $N$ is prime, then the ideal $\p = (N : M)$
is a prime ideal of $R$. In this circumstances,  $N$ is said to be
\emph{$\p$-prime} (see \cite{lu84}). A submodule $Q$ of an $R$-module $M$ is said to be
\emph{primary} if $Q\neq M$ and if $rm \in Q$, where $r \in R$ and $m
\in M$ implies that either $m \in Q$ or $r \in \q=\sqrt{(Q : M)}$.
If $Q$ is primary, then $(Q : M)$ is a primary ideal of $R$. In this case we say that $Q$ is \emph{$\q$-primary}, where $\q = \sqrt{(Q
: M)}$ is a prime ideal of $R$. The set of all prime submodules
of an $R$-module $M$ is called the \emph{prime spectrum} of $M$ and denoted by
$\Spec(M)$. Similarly, the collection of all $\p$-prime submodules
of an $R$-module $M$  is designated by $\Spec_{\p}(M)$ for any $\p\in \Spec(R)$. We
remark that $\Spec(\textbf{0}) = \emptyset$ and that $\Spec(M)$ may
be empty for some nonzero module $M$. For example, the
$\mathbb{Z}(p^{\infty})$ as a $\mathbb{Z}$-module has no prime
submodule for any prime integer $p$ (see \cite{lu95}). Such a
module is said to be \emph{primeless}.

An $R$-module $M$ is called
\emph{primeful} if either $M =(\textbf{0})$ or $M\neq (\textbf{0})$ and
the map $\Phi : \Spec(M)\rightarrow \Spec(R/\Ann(M))$ defined by
$\Phi(P) = (P : M)/\Ann(M)$ for every $P\in \Spec(M)$, is
surjective (see \cite{lu07}). The set of all maximal submodules
of an $R$-module $M$ is denoted by $\Max(M)$. The \emph{Jacobson radical} $\Rad(M)$ of a
module $M$ is the intersection of all its maximal submodules. $\Rad(M) = M$ when $M$ has no any maximal submodule. By
$N\leq M$ we mean that $N$ is a submodule of $M$. Let $\p$ be a
prime ideal of $R$, and $N\leq M$. By the \emph{saturation of $N$ with
respect to $\p$}, we mean the contraction of $N_{\p}$ in $M$ and
designate it by $S_{\p}(N)$ and we say $N$ is \emph{saturated with
respect to} $\p$ if $S_{\p}(N)=N$ (see \cite{lu03}).

An $R$-module $M$ is called a
\emph{multiplication} module if every submodule $N$ of $M$ is of the
form $IM$ for some ideal $I$ of $R$. For any submodule $N$ of an $R$-module $M$
we define $V^M(N)$ to be the set of all prime submodules of $M$
containing $N$. The \emph{radical} of $N$ defined to be the intersection
of all prime submodules of $M$ containing $N$ and denoted by
$\rad_M(N)$ or briefly $\rad(N)$. $\rad_M(N) = M$ when $M$ has no any prime
submodule containing $N$. In particular, $\rad(\textbf{0}_M)$ is
the intersection of all prime submodules of $M$. If $V^M(N)$ has
at least one minimal member with respect to the inclusion, then every minimal member in this form
is called a \emph{minimal prime submodule of $N$}
or a \emph{prime submodule minimal over $N$}. A minimal prime submodule
of $(\textbf{0})$ is called  \emph{minimal prime submodule of $M$}. A
quasi-prime submodule $N$ of an $R$-module $M$ is called \emph{minimal quasi-prime}
if, for any quasi-prime $K$ of $M$ such that $K\subseteq N$, this
is the case that $K=N$. An $R$-module $M$ is said to be \emph{semiprimitive}
(resp. reduced) if the intersection of all maximal (resp. prime)
submodules of $M$ is equal to zero. A submodule $N$ of an $R$-module $M$ is
said \emph{quasi-semiprime} if it is an intersection of quasi-prime
submodules. We recall that an $R$-module $M$ is \emph{co-semisimple} in case every submodule of $M$ is the intersection of maximal submodules (see \cite[p.122]{ful92}). Every proper submodule of a co-semisimple module is a quasi-semiprime submodule.

In Section $2$, we obtain some properties of quasi-prime submodules.
In this section the relations between quasi-prime submodules of a module $M$ and
quasi-prime submodules of localizations of $M$ are studied.
We also investigate the quasi-primeful modules and we  apply them to
develop topological properties of $q\Spec(M)$.
 We show in Theorem~\ref{mapa} that an $R$-module $M$ is quasi-primeful whenever
$R$ is a $PID$ and $M$ is finitely generated, or $R$ is Laskerian and $M$ is a locally free $R$-module.
We study some main properties of quasi-primeful modules in Proposition \ref{paracor} and also the quasi-prime-embedding modules are studied in Theorem \ref{freet0}. It is shown that an $R$-module $M$ is top in the cases $R$ is a one dimensional Noetherian  domain and either $M$ is weak multiplication  or for every prime ideal $\p\in \Spec(R)$, $|\Spec_{\p}(M)|\leq 1$ and $S_{(0)}(\textbf{0})\subseteq \rad(\textbf{0})$.
In Section $3$, we introduce a topology on the set of
quasi-prime submodules in such a way that the Zariski topology (see
\cite{lu99}) is a subspace of this topology and some concerned properties
 are given. An $R$-module whose developed Zariski topology is $T_0$, irreducible or Noetherian
is studied in Section 3.

\section{SOME PROPERTIES OF QUASI PRIME SUBMODULES}

In this section we introduce the notion of quasi-prime submodule and find some properties of it. We also introduce the notions of quasi-primeful and quasi-prime-embedding modules and we use them in the next section.

\begin{defn}
A proper submodule $N$ of an $R$-module $M$ is called \emph{quasi-prime} if $(N:_R M)$ is a quasi-prime ideal of $R$.
\end{defn}

We define the \emph{quasi-prime spectrum} of an $R$-module $M$ to be the set of all quasi-prime submodules of $M$
 and denote it by $q\Spec^R(M)$. If there is no ambiguity we write only $q\Spec(M)$ instead of $q\Spec^R(M)$.
 For any $I\in q\Spec(R)$, the collection of all quasi-prime submodules $N$ of $M$ with $(N:M)=I$ is designated by
 $q\Spec_I(M)$. We say that $R$ is a \emph{serial ring} if the set of all ideals of $R$ is linearly ordered. Recall that a ring $R$ is
  said to be  \emph{arithmetical},  if for any maximal ideal $\p$ of $R$, $R_{\p}$ is a serial ring (see \cite{Jen66}). Recall that a module $M$ is said to be a \emph{Laskerian module}, if every proper submodule of $M$ has a
primary decomposition. We know that every Noetherian module is  Laskerian.

\begin{rem}\label{6} (See \cite{az08}, \cite{hen02} and \cite{Jen66})
Let $I$ be an ideal in a ring $R$ and S be a multiplicatively closed subset of $R$. Then
\begin{enumerate}
\item If $I$ is quasi-prime, then $I$ is irreducible;
\item If $R$ is a Laskerian ring, then every quasi-prime ideal is a primary ideal;
  \item If $I$ is a prime ideal, then $I$ is quasi-prime;
  \item Every proper ideal of a serial ring is quasi-prime;
  \item If $IR_S$ is a quasi-prime ideal of $R_S$, then $IR_S \cap R$ is a quasi-prime ideal of $R$;
  \item If $I$ is a quasi-prime and primary ideal of $R$ such that $I \cap S =\emptyset$, then $IR_S$ is a quasi-prime ideal of $R_S$;
  \item If $R$ is an arithmetical ring, $I$ is irreducible if and only if $I$ is quasi-prime;
  \item In an arithmetical ring $R$ any primary ideal is irreducible;
  \item If $R$ is a Dedekind domain, then $I$ is quasi-prime if and only if $I$ is a primary ideal.
\end{enumerate}
\end{rem}

\begin{rem}\label{remaux}
Let $M$ be an $R$-module.
\begin{enumerate}
  \item By \cite[Proposition 4]{lu84}, every maximal submodule of an $R$-module $M$ is prime and by Remark~\ref{6}, every prime submodule of $M$ is a quasi-prime submodule. Therefore, $\Max(M)\subseteq \Spec(M)\subseteq  q\Spec(M)$.
      So, $q\Spec(M)\neq\emptyset$ if $M$ is not primeless.
  \item Consider $M=\mathbb{Z}\oplus \mathbb{Z}$ as a $\mathbb{Z}$-module and $N=(2,0)\mathbb{Z}$ is the submodule of $M$ generated by $(2, 0) \in M$. Then $(N:M)=(0)\in \Spec(\mathbb{Z})$, i.e., $N\in  q\Spec(M)$ though $N$ is not a $(0)$-prime submodule of $M$. Thus in general, a quasi-prime submodule need not be a prime submodule, i.e., $\Spec(M)\neq  q\Spec(M)$.
  \item As another example, we consider the faithful torsion $\mathbb{Z}$-module $M=\bigoplus_{p} \mathbb{Z}/{p\mathbb{Z}}$, where $p$ runs through the set of all prime integers. Let $N=(\textbf{0})$ and $\p=(0)$. Then $(N:M)=(\textbf{0}:M)=\Ann(M)=(0)$. Hence, $N\in q\Spec(M)$. However, $N$ is not a prime submodule by \cite[Result 2]{lu03}, because $S_{\p}(N)=S_{(0)}(\textbf{0})=M$.
\end{enumerate}
\end{rem}

An $R$-module $M$ is called a \emph{fully prime module} if every proper submodule is a prime submodule. In \cite[Proposition 1.10]{bkk04}, the authors give several equivalent conditions for an $R$-module $M$ to be fully prime, for example, $M$ is a fully prime $R$-module if and only if $\Ann(M)$ is a maximal ideal, i.e., if and only if $M$ is a homogeneous semisimple module (i.e., a direct sum of isomorphic simple $R$-modules).

\begin{lem}\label{lemaux}\label{lemaux2}
Let $J\in q\Spec(R)$, $\p\in \Spec(R)$, $I$ be a proper ideal of $R$ and $M$ be an $R$-module with submodule $N$. Let $S$ be a multiplicatively closed subset of $R$.
\begin{enumerate}
  \item If $N\in q\Spec_J(M)$, then $(N:M)M\in q\Spec_J(M)$;
  \item If $\{N_\lambda\}_{\lambda\in \Lambda}$ is a family of quasi-prime submodules with $(N_\lambda:M)=J$ for each $\lambda\in \Lambda$, then $\cap_{\lambda\in \Lambda}N_\lambda\in q\Spec_J(M)$;
  \item If $M$ is a fully prime module, then every proper submodule of $M$ is quasi-prime. In particular, every proper subspace of a vector space over a field is quasi-prime;
  \item If $R$ is a serial ring, then every proper submodule of $M$ is quasi-prime;\label{n1}
  \item Let $N$ be a quasi-prime submodule of the $R_S$-module $M_S$. Then $N\cap M$ is a quasi-prime submodule of $M$. So, $\{N\cap M \,|\; N\in q\Spec(M_S)\}\subseteq  q\Spec(M)$;\label{n2}
  \item Let $R$ be Laskerian and $M$ be a finitely generated $R$-module. If $N$ is a quasi-prime submodule of $M$ and $\sqrt{(N:M)}\cap S=\emptyset$, then $N_S$ is a quasi-prime submodule of $M_S$;\label{n3}
  \item Let $R$ be an arithmetical ring. Then every primary submodule of $M$ is quasi-prime;
  \item Let $R$ be an arithmetical ring. If $\p\in V^R(I)$, then $S_{\p}(I)$ is a quasi-prime ideal of $R$. Moreover, if $R$ is Laskerian, then $S_{\p}(I)$ is primary and $\p$ is a minimal prime ideal over $I$;\label{n4}
  \item Let $R$ be an arithmetical ring. Let $N$ be a submodule of $M$ and $\p\in Supp(M/N)$. Then $S_{\p}(N)$ is a quasi-prime submodule of $M$. Therefore, every proper saturated submodule $N$ w.r.t $\p$, is a quasi-prime submodule of $M$;
  \item Let $R$ be an arithmetical ring and let $M$ be a finitely generated $R$-module. If $N$ is a quasi-prime submodule of $M$ and $\p\in V^R(N:M)$, then $N_{\p}$ is a quasi-prime submodule of $M_{\p}$.
\end{enumerate}
\end{lem}
\begin{proof}

\begin{enumerate}
  \item[(1)-(3)] are clear.
  \item[(4)] Every proper ideal of $R$ is quasi-prime by Remark~\ref{6}.

  \item[(5)]  One can obtain that $((N\cap M):_R M)=(N:_{R_S} M_S)\cap R$. Now, let $I:=(N:_{R_S} M_S)\cap R$. Then $IR_S=(N:_{R_S} M_S)$ is a quasi-prime ideal of $R_S$ by assumption. By Remark~\ref{6}, $I$ is a quasi-prime ideal of $R$ so, $N\cap M$ is a quasi-prime submodule of $M$.

  \item[(6)]~By assumption, $(N:_R M)$ is a quasi-prime ideal and since $R$ is Laskerian,  $(N:_R M)$ is primary. By Remark~\ref{6} and \cite[p. 152, Proposition 8]{nor68}, $(N_S:_{R_S} M_S)=(N:_R M)R_S$ is a quasi-prime ideal of $R_S$. So, $N_S$ is a quasi-prime submodule of $M_S$.

  \item[(7)] Let $N$ be a primary submodule of $M$. Then $(N:_R M)$ is a primary ideal of $R$, so is quasi-prime by Remark~\ref{6}. Hence, $N\in q\Spec(M)$.

  \item[(8)] $IR_{\p}$ is a proper ideal of $R_{\p}$. But $R_{\p}$ is a serial ring. Thus by Remark~\ref{6}, $IR_{\p}$ is a quasi-prime ideal of $R_{\p}$ and therefore $S_{\p}(I)=IR_{\p} \cap R$ is quasi-prime  by Remark~\ref{6}. If $R$ is Laskerian, then by Remark~\ref{6}, $S_{\p}(I)$ is primary. Let $\q$ be a prime ideal of $R$ such that $I\subseteq \q \subseteq \p$. Then $$S_\p(I)\subseteq S_\p(\q) \subseteq S_\p(\p).$$ By definition, $S_\p(\q)=\q$ and $S_\p(\p)=\p$. Since $S_\p(I)$ is a $\p$-primary ideal of $R$, we have $$\p=\sqrt{S_\p(I)}\subseteq \sqrt{S_\p(\q)} \subseteq \sqrt{S_\p(\p)}=\p.$$ Therefore, $\q=\p$ and $\p$ is minimal prime ideal over $I$.
  \item[(9)] Since $\p\in Supp(M/N)$, $N_{\p}\neq M_{\p}$. By assumption $R_{\p}$ is a serial ring. By part~(\ref{n1}), $N_{\p}$ is a quasi-prime submodule of $M_\p$. By part~(\ref{n2}), $S_{\p}(N)=N_{\p}\cap M$ is a quasi-prime submodule of $M$. The last assertion follows from \cite[Result 2]{lu03}.

  \item[(10)] We have $(N_\p : M_\p) = (N :M)_\p\subseteq \p R_\p$ and $R_\p$ is a serial ring. So, $N_p$ is a quasi-prime submodule of $M_p$.
\end{enumerate}
\end{proof}

It is shown in \cite[Proposition 2.1]{az03} that $R$ is a field if every proper submodule of $M$ is a prime submodule of $M$ and $S_{(0)}(\textbf{0})\neq M$. In the following, we give an example  that shows it is not the case for any quasi-prime submodule.

\begin{eg}\label{egzpi}
(1)  Every proper submodule of the $\mathbb{Z}$-module $M=\mathbb{Z}(p^\infty)$ is a quasi-prime submodule, in which $p$ is a prime integer. For, $(L:_{\mathbb{Z}}M)=(0)$ where $L$ is a submodule of $M$ (see \cite[p.~3745]{lu95}).

(2) Let $R$ be an integral domain which is not a field and $K$ be the field of quotients of $R$. Then every proper submodule of $K$ is a quasi-prime submodule. Since $xK=K$ for every nonzero element $x\in R$, $(N:K)=(0)$ for every proper submodule $N$ of $K$.
\end{eg}




\begin{thm}\label{thm2.16}
Let $M$ be a finitely generated $R$-module and let $I$ be a primary quasi-prime ideal of $R$. If $S$ is a multiplicatively closed subset of $R$ such that $I\cap S=\emptyset$, then the map $N \mapsto N_S$ is a surjection from $q\Spec_I(M)$ to $q\Spec_{IR_S}(M_S)$.
\end{thm}
\begin{proof}
Let $N\in q\Spec_I(M)$. Since $M$ is finitely generated and $I\cap S=\emptyset$ we have $IR_S=(N:_R M)R_S=(N_S:_{R_S} M_S)\neq R_S$. By Remark~\ref{6}, $IR_S$ is a quasi-prime ideal of $R_S$. Therefore, $N_S$ is a quasi-prime submodule of $M_S$. Let $L$ be a quasi-prime submodule of $M_S$ with $(L:_{R_S}M_S)=IR_S$. By Lemma~\ref{lemaux}(\ref{n2}), $L\cap M$ is a quasi-prime submodule of $M$. Moreover, using that $I$ is primary we have $$I=IR_S\cap R=(L:_{R_S} M_S)\cap R=((L\cap M):_R M).$$ So, $L\cap M$ is a quasi-prime submodule of $M$.
\end{proof}

\begin{cor}
Let $M$ be a finitely generated $R$-module and $\p\in \Spec(R)$.
\begin{enumerate}
  \item Let $I$ be a $\p$-primary quasi-prime ideal of $R$. Then the map $N \mapsto N_\p$ is a surjection from $q\Spec_I(M)$ to $q\Spec_{IR_{\p}}(M_{\p})$.
  \item The map $N \mapsto N_\p$ is a surjection from $q\Spec_{\p}(M)$ to $q\Spec_{\p R_{\p}}(M_{\p})=\Spec_{\p R_{\p}}(M_{\p})$.
  \item Let $N$ be a quasi-prime submodule of $M$ with $(N:M)=\p$. Then $S_{\p}(N)$ is a prime submodule minimal over $N$ and any other $\p$-prime submodule of $M$ containing $N$, must contain $S_{\p}(N)$.
\end{enumerate}
\end{cor}
\begin{proof}
(1) and (2) follows from Theorem~\ref{thm2.16}. For establish (3), note that by part~(2), $N_{\p}$ is a $\p R_{\p}$-prime submodule of $M_{\p}$ and by \cite[Proposition~1]{lu95}, $S_{\p}(N)=N_{\p}\cap M$ is a $\p$-prime submodule of $M$. Now the results follows from \cite[Result~3]{lu03}.
\end{proof}

\begin{defn}
Let $M$ be an $R$-module. For a submodule $N$ of $M$ we define
\begin{eqnarray*}
  D^M(N) &=& \{ L\in  q\Spec(M)\mid (L:M)\supseteq (N:M)\}, \\
  \Omega^M(N) &=& \{L\in  q\Spec(M)\mid L\supseteq N\} .
\end{eqnarray*}
\end{defn}

If there is no ambiguity we write $D(N)$ (resp. $\Omega (N)$) instead of $D^M(N)$ (resp. $\Omega^M(N)$).

\begin{lem}
Let $M$ be an $R$-module with $ q\Spec(M)=\emptyset$. Then $\p M = M$ for every maximal
ideal $\p$ of $R$. On the other hand, if $IM = M$ for every $I\in D^R(\Ann(M))$, then $q\Spec(M)=\emptyset$.
\end{lem}

\begin{defn}
When $ q\Spec(M) \neq\emptyset$, the map $\psi :  q\Spec(M)\rightarrow
 q\Spec(R/\Ann(M))$ defined by $ \psi(L) = (L : M)/\Ann(M)$ for every
$L\in  q\Spec(M)$, will be called the natural map of $ q\Spec(M)$. An $R$-module $M$ is called \emph{quasi-primeful} if either $M =(\textbf{0})$ or $M\neq (\textbf{0})$ and has a surjective natural map.
\end{defn}

\begin{eg}\label{paralleleg}
Let $\Sigma:=q\Spec(\mathbb{Z})\setminus \{(0)\}$. Consider
the $\mathbb{Z}$-module $M=\bigoplus_{I\in \Sigma}
\mathbb{Z}/{I}$. We will show that $M$ is  a quasi-primeful
$\mathbb{Z}$-module. Note that $(\textbf{0}:M)=\Ann(M)=(0)$. So,
$(\textbf{0})\in q\Spec_{(0)}(M)$. On the other hand, for each
nonzero quasi-prime ideal $I$ of $\mathbb{Z}$, we have
$(IM:M)=I\in q\Spec(\mathbb{Z})$. This implies that $IM\in
q\Spec_I(M)$. We conclude that $M$ is  a quasi-primeful
$\mathbb{Z}$-module.
\end{eg}

Let $Y$ be a subset of $ q\Spec(M)$ for an $R$-module $M$. We will
denote the intersection of all elements in $Y$ by $\Im(Y )$.

\begin{prop}\label{free}
Let $F$ be a free $R$-module and $I$ be a quasi-prime ideal of $R$. Then
\begin{enumerate}
  \item $IF$ is a quasi-prime submodule, i.e., $F$ is quasi-primeful;
  \item $IF=\Im(q\Spec_I(F))$;
  \item If $F$ has primary decomposition for submodules, then $I$ is primary.
\end{enumerate}
\end{prop}
\begin{proof}
(1) Since $F$ is free we have $I=(IF:F)$, so that $IF$ is a quasi-prime submodule. (2) This is clear by (1). For (3), Let $\cap_{i=1}^n Q_i$ be a primary decomposition of $IF$, where each $Q_i$ is a $\p_i$-primary submodule of $F$. Then $I=(IF:F)=\cap_{i=1}^n (Q_i:F)$. Since $I$ is quasi-prime, $I=(Q_j:F)$ for some $1\leq j\leq n$. Hence, $I$ is primary since $Q_j$ is a primary submodule.
\end{proof}

\begin{lem}\label{directsum}
Let $M$, $M_1$, $M_2$ be $R$-modules such that $M = M_1\oplus M_2$ and $I\in D^R(\Ann(M))$. If $N\in q\Spec_I(M_1)$ (resp. $N\in q\Spec_I(M_2)$), then $N\oplus M_2\in q\Spec_I(M)$ (resp.
$M_1\oplus N\in q\Spec_I(M)$). In particular, every direct sum of a finite number of quasi-primeful $R$-modules is quasi-primeful over $R$.
\end{lem}
\begin{proof}
This is straightforward and we omit it.
\end{proof}

\begin{thm}\label{mapa}
Let $M$ be an $R$-module. Then $M$ is quasi-primeful in each of the following cases:
\begin{enumerate}
  \item $R$ is $PID$ and $M$ is finitely generated;
  \item $R$ is Dedekind domain and $M$ is faithfully flat;
  \item $R$ is Laskerian and $M$ is locally free.
\end{enumerate}
\end{thm}
\begin{proof}
(1) Let $N$ be a cyclic submodule of $M$ and $I\in D^R(\Ann(N))$. Then $N=R/\Ann(m)$ for some $m\in N$ and $(I/\Ann(N):N)=I$. Hence, $N$ is quasi-primeful. It is well-known that a finitely generated module over a $PID$ is finite direct sum of cyclic submodules. Hence, in the light of Lemma~\ref{directsum},  $M$ is quasi-primeful.
(2) Let $J\in q\Spec(R)$. Since $M$ is faithfully flat, $JM\neq M$ and by Remark~\ref{6}, $J$ is primary. So, $JM$ is a primary submodule by ~\cite[Theorem 3]{lu84}, and $(JM:M)=J$ is a quasi-prime ideal of $R$, i.e., $JM$ is quasi-prime.
(3) Let $I\in D(\Ann(M))$. Since $R$ is Laskerian, $\p:=\sqrt{I}$ is a prime ideal of $R$ and $IR_{\p}$ is a quasi-prime ideal of $R_{\p}$ by Remark~\ref{6}(6). Since $M_{\p}$ is a free $R_{\p}$-module, there exists a quasi-prime submodule $N$ of $M_{\p}$  such that $(N:_{R_{\p}}M_{\p})=IR_{\p}$ by Proposition~\ref{free}. Now, $(N\cap M:M)=IR_{\p}\cap R=I$ by Lemma~\ref{lemaux}. This implies that $M$ is quasi-primeful.
\end{proof}

We note that not every quasi-primeful module is finitely generated. For example, every (finite or infinite dimensional) vector space is quasi-primeful.

\begin{rem}\label{multi}(See \cite[Theorem 3.1]{zb88})
Let $M$ be a faithful multiplication module over $R$. Then $M$ is finitely generated if and only if $\m M\neq M$ for every maximal ideal $\m$ of $R$.
\end{rem}

\begin{prop}\label{pva}\label{paracor}
Let $M$ be a nonzero quasi-primeful $R$-module.
\begin{enumerate}
  \item Let $I$ be a radical ideal of $R$. Then $(IM : M )= I$ if and only if $\Ann(M)\subseteq I$;
  \item $\p M\in q\Spec(M)$ for every $\p\in V(\Ann(M))$;
  \item $\p M\in \Spec_{\p}(M)$ for every $\p\in V(\Ann(M))\cap \Max(R)$;
  \item If $\dim(R)=0$, then $M$ is primeful;
  \item If $M$ is multiplication, then $M$ is finitely generated.
\end{enumerate}
\end{prop}
\begin{proof}
(1) The necessity is clear. For sufficiency, we note that $\Ann(M)\subseteq I=\cap_i~\p_i$, where $\p_i$ runs through $V^R (I)$ since $I$ is a radical ideal. On the other hand, $M$ is quasi-primeful and $\p_i\in D(\Ann(M))$ so, there exists a quasi-prime submodule $L_i$ such that $(L_i:M)=\p_i$. Now, we obtain that $$I \subseteq (IM : M) = ((\cap_i ~\p_i)M : M)\subseteq \cap_i~(\p_iM : M)\subseteq \cap_i~(L_i: M) = \cap_i~ \p_i = I.$$ Thus $(IM : M) = I$. (2) and (3) follows from part~(1). For (4), let $\p\in V^R(\Ann(M))$. Then by part~(3), $\p M\neq M$ and  by \cite[Result 3]{lu07}, $M$ is primeful. (5) Since $M$ is a faithful multiplication module over $R/\Ann(M) = \bar{R}$ and $\bar{\m}M\neq M$ for every $\bar{\m}\in \Max(\bar{R})$ by (3), $M$ is finitely generated over $\bar{R}$ by Remark~\ref{multi}. Hence, $M$ is finitely generated over $R$.
\end{proof}

\begin{cor}
Let $M$ be an $R$-module.
\begin{enumerate}
\item Let $M$ be a quasi-primeful $R$-module. If $I$ is an ideal of R contained in the Jacobson radical $\Rad(R)$
such that $IM = M$, then $M = (\textbf{0})$.
  \item Let $R$ be a $PID$ and $M$ be torsion-free. Then $M$ is quasi-primeful if $pM\neq M$ for every irreducible element $p\in R$.
  \item If $M$ is faithful quasi-primeful, then $M$ is flat if and only if $M$ is faithfully flat.
  \item If $M$ is projective and $R$ is Laskerian, then $M$ is quasi-primeful.
\end{enumerate}
\end{cor}
\begin{proof}
(1) Suppose that $M\neq(\textbf{0})$. Then $\Ann(M)\neq R$. If $\m$ is any maximal ideal containing $\Ann(M)$, then $I \subseteq \Rad(R)\subseteq \m$ and $IM = M = \m M$ whence $(\m M : M) = R \neq \m$, a contradiction to Proposition~\ref{pva}. (2) If for every irreducible element $p\in R$, $pM\neq M$, then $M$ is faithfully flat and by Theorem~\ref{mapa}, $M$ is quasi-primeful. (3) The sufficiency is clear. Suppose that $M$ is flat. By Proposition~\ref{pva}, for every $\p\in \Max(R)\subseteq D(0)$, $\p M\neq M$. This implies that $M$ is faithfully flat. (4) Since every projective module is locally free, by Theorem~\ref{mapa}, $M$ is quasi-primeful.
\end{proof}

\begin{eg}
The $\mathbb{Z}$-module $\mathbb{Q}$ is a flat and faithful, but not faithfully flat. So, $\mathbb{Q}$ is not quasi-primeful.
\end{eg}

We give an elementary example of a module which is not quasi-primeful. If $R$ is a domain, then an $R$-module $M$ is \emph{divisible} if $M = rM$ for all nonzero elements $r\in R$. We note that every injective module is divisible.

\begin{prop}\label{divi}
Let $R$ be a domain which is not a field. Then every nonzero divisible $R$-module is not quasi-primeful.
\end{prop}
\begin{proof}
By assumption $\Ann(M)=(0)$ and there exists a nonzero prime ideal $\p$ of $R$. Hence $\p\in V^R(\Ann(M))$ and $\p M = M$. Therefore, $M$ is not quasi-primeful by Proposition~\ref{paracor}.
\end{proof}

\begin{prop}
Let $R$ be a domain over which every module is quasi-primeful. Then $R$ is a field.
\end{prop}
\begin{proof}
Suppose that $R$ is not a field. Then its field $K$ of quotients is a nonzero
divisible $R$-module. Hence, $K$ is not quasi-primeful over $R$ by Proposition~\ref{divi}, which is a contradiction to the definition of $R$.
\end{proof}

An $R$-module $M$ is called \emph{weak multiplication} if
$\Spec(M) =\emptyset$ or for every prime submodule $N$ of $M$, we have $N = IM$,
where $I$ is an ideal of $R$. One can easily show that if $M$ is a weak multiplication module, then $N = (N : M)M$ for every prime submodule $N$ of $M$ (\cite{Abu95} and \cite{az03}).
As is seen in \cite{Abu95}, $\mathbb{Q}$ is a weak multiplication $\mathbb{Z}$-module which is not a multiplication module.

\begin{defn}
An $R$-module $M$ is called \textit{quasi-prime-embedding}, if the natural map $\psi :  q\Spec(M)\rightarrow
 q\Spec(R/\Ann(M))$ is injective.
\end{defn}

We will show that every cyclic module is quasi-prime-embedding (Corollary~\ref{maincor}).
Thus any ring $R$ as $R$-module is quasi-prime-embedding.

\begin{prop}\label{t0inj}
The following statements are equivalent for any $R$-module $M$:
\begin{enumerate}
  \item $M$ is quasi-prime-embedding;
  \item If $D (L) = D (N)$, then $L = N$, for any $L, N\in  q\Spec(M)$;
  \item $|q\Spec_I(M)| \leq 1$ for every $I\in  q\Spec(R)$.
\end{enumerate}
\end{prop}
\begin{proof}
$(1)\Rightarrow (2)$ Let $D(L) =  D(N)$. Then $(L : M) = (N : M )$. Now by (1), $L = N$. $(2)\Rightarrow (3)$ Suppose that $L, N\in q\Spec_I(M)$ for some $I\in q\Spec(R)$. Hence $(L : M) = (N : M )= I$ and so, $D(L) =  D(N)$. Thus, $L = N$ by (2). $(3)\Rightarrow (1)$ Let $\bar{I}:=\psi(L)=\psi(N)$. Then $I=(L : M) = (N : M )$. By (3), $L = N$, and so $\psi$ is injective.
\end{proof}

\begin{cor}\label{maincor}
Consider the following statements for an $R$-module $M$ :
\begin{enumerate}
  \item $M$ is multiplication;
  \item $M$ is quasi-prime-embedding;
  \item $M$ is weak multiplication;
  \item $|\Spec_{\p}(M)|\leq 1$ for every prime ideal $\p$ of $R$;
  \item $M/pM$ is cyclic for every maximal ideal $\p$ of $R$.
\end{enumerate}
Then $(1)\Rightarrow (2)\Rightarrow (3)\Rightarrow (4)\Rightarrow (5)$. Further, if $M$ is finitely generated, then $(5)\Rightarrow (1)$.
\end{cor}
\begin{proof}
$(1)\Rightarrow (2)$ Let $D(N)=D(L)$ for  $N$, $L\in  q\Spec(M)$. Then $(N:M)=(L:M)$ and since $M$ is multiplication, $N=L$. Therefore, (2) follows from Proposition~\ref{t0inj}. $(2)\Rightarrow (3)$ Let $P$ be a $\p$-prime submodule of $M$. By Lemma~\ref{lemaux}, $(P:M)M\in q\Spec_{\p}(M)$. Combining this fact with Proposition~\ref{t0inj}, we obtain that $P=(P:M)M$. This yields $M$ is weak multiplication. $(3)\Rightarrow (4)$ The case $\Spec_{\p}(M)=\emptyset$ is trivially true. Let $P$, $Q\in \Spec_{\p}(M)$ for some prime ideal $\p$ of $R$. Then $(P:M)=(Q:M)$. Therefore $P=(P:M)M=(Q:M)M=Q$. The $(4)\Rightarrow (5)$ and last statement is true due to ~\cite[Theorem~3.5]{mms97}.
\end{proof}

An $R$-module $M$ is called \emph{locally cyclic} if $M_{\p}$ is a cyclic module over the local ring $R_{\p}$ for every prime ideal $\p$ of $R$. Multiplication modules are locally
cyclic (see \cite[Theorem 2.2]{zb88}).

\begin{thm}\label{freet0}
Let $M$ be an $R$-module and let $S$ be a multiplicatively closed subset of $R$.
\begin{enumerate}
  \item If $M$ is Laskerian quasi-prime-embedding, then every quasi-prime submodule of $M$ is primary (see \cite[Theorem 2.1]{az08}).
  \item Let $R$ be a serial ring. Then $M$ is multiplication if and only if $M$ is quasi-prime-embedding.\label{serial}
  \item If $M$ is quasi-prime-embedding, then $S^{-1}M$ is also a quasi-prime-embedding $S^{-1}R$-module.\label{t0local}
  \item If $M$ is free, Then $M$ is quasi-prime-embedding if and only if $M$ is cyclic.\label{t0fr}
  \item If $M$ is projective quasi-prime-embedding, then $M$ is locally cyclic.
  \item If $R$ is an arithmetical ring and $M$ is quasi-prime-embedding, then $M$ is locally cyclic.\label{arith}
  \item Let $R$ be a semi-local arithmetical ring. Then $M$ is cyclic if and only if $M$ is quasi-prime-embedding.
  \item A finitely generated module $M$ is locally cyclic if and only if $M$ is multiplication if and only if $M$ is quasi-prime-embedding.
  \item Let $R$ be a Dedekind domain and $M$ be a non-faithful quasi-prime-embedding \, $R$-module. Then $M$ is cyclic.
\end{enumerate}
\end{thm}
\begin{proof}
(1) Let $P$ be a quasi-prime submodule of $M$ and $\bigcap_{i=1}^m N_i$ be a primary decomposition for $P$. Since $P$ is quasi-prime, $$(N_j:M)\subseteq (P:M)=\bigcap_{i=1}^m (N_i:M)\subseteq (N_j:M)$$ for some $1\leq j\leq m$. Hence, $N_j$ is a quasi-prime submodule and by Proposition~\ref{t0inj}, $P=N_j$.\\
(2) The necessity follows from Corollary~\ref{maincor}.  Let $N$ be a proper submodule of $M$. By Lemma~\ref{lemaux}, $N$ and $(N:M)M$ are quasi-prime submodules of $M$. Therefore, $N=(N:M)M$ by Proposition~\ref{t0inj}, and so $M$ is multiplication.\\
(3) Use Lemma~\ref{lemaux} and Proposition~\ref{t0inj}.\\
(4) If $M$ is cyclic, then $M$ is quasi-prime-embedding by Corollary~\ref{maincor}. We assume that $M$ is quasi-prime-embedding and $M$ is not cyclic. Hence, $M=\oplus_{i\in I} R$, where $|I|> 1 $. Let $\p\in q\Spec(R)$ and $\alpha, \beta$ be two distinct elements of $I$. It is easy to see that $$N=\p\oplus(\oplus_{\substack{i\in I\\i\neq \alpha}} R) \quad \text{  and  }\quad  L=\p\oplus(\oplus_{\substack{i\in I\\i\neq \beta}} R)$$ are two distinct quasi-prime submodules of $M$  with $(N:M)=(L:M)=\p$. By Proposition~\ref{t0inj}, $N=L$, a contradiction.\\
(5) Let $\p\in \Spec(R)$. Then by (\ref{t0local}), $M_{\p}$ is quasi-prime-embedding. On the other hand, $M_{\p}$ is a free $R_{\p}$-module. Hence,  $M_{\p}$ is a cyclic $R_{\p}$-module by (\ref{t0fr}).\\
(6) For each $\p\in \Spec(R)$, $R_P$ is a serial ring by \cite[Theorem~1]{Jen66}, and $M_{\p}$ is quasi-prime-embedding by (\ref{t0local}). By (\ref{serial}), $M_{\p}$ is a multiplication $R_{\p}$-module. Therefore, $M_{\p}$ is cyclic, since $R_{\p}$ is a quasi-local ring.\\
(7) Let $\m_1, \cdots , \m_t$ be all maximal ideals of $R$. By (\ref{arith}), $M_{\m_i}$ is a cyclic $R_{\m_i}$-module for each $i$. Hence, $M$ is cyclic by \cite[Lemma 3]{Bar81}. Other side is true by Corollary~\ref{maincor}.\\
(8) Use \cite[Proposition 5]{Bar81} and Corollary~\ref{maincor}.\\
(9) By assumption there exist only finitely many prime (maximal) ideal containing $\Ann(M)$. So, by (\ref{arith}), and  \cite[Lemma 3]{Bar81}, $M$ is cyclic.
\end{proof}

A submodule $S$ of an $R$-module $M$ will be called \emph{semiprime} if $S$ is an intersection of prime submodules. A prime submodule $K$ of $M$ is said to be \emph{extraordinary} if whenever $N$ and $L$ are semiprime submodules of $M$ with $N\cap L \subseteq K$, then $N\subseteq K$ or $L\subseteq K$. An $R$-module $M$ is said to be a \emph{top module} if every prime submodule of $M$ is extraordinary. Every multiplication or locally cyclic module is a top module (see \cite{mms97}).
Corollary~\ref{maincor} and Theorem~\ref{freet0} are very interesting for us, because there is a close relationship between those and top modules. We find the relations
between parts (1)-(4) of Corollary~\ref{maincor} and top modules. By \cite[Theorem~3.5]{mms97}, every multiplication module is top. So we consider part (2) of Corollary~\ref{maincor}. By Theorem~\ref{freet0}, every projective quasi-prime-embedding module and every quasi-prime-embedding module over arithmetical ring is locally cyclic,  so is top due to \cite[Theorem~4.1]{mms97}. In the next theorem we will show the relationship between part~(3) and part~(4) of Corollary~\ref{maincor} and top modules.

\begin{thm}\label{maintop}
Let $R$ be a one dimensional Noetherian domain and let $M$ be a nonzero $R$-module. Then $M$ is a top module in each of the following cases:
\begin{enumerate}
  \item $M$ is weak multiplication.
  \item For every prime ideal $\p\in \Spec(R)$, $|\Spec_{\p}(M)|\leq 1$ and $S_{(0)}(\textbf{0})\subseteq \rad(\textbf{0})$.
\end{enumerate}
\end{thm}
\begin{proof}
\begin{enumerate}

\item  Let $P$ be a $\p$-prime submodule of $M$ and let $N$ and $L$ be non-zero semiprime submodules of $M$ such that $N \cap L \subseteq P$. It is enough to show that $N \subseteq P$ or $L\subseteq P$. If $(N:M)$ or $(L:M)$ $\not\subseteq (P:M)$, then $L\subseteq P$ or $N\subseteq P$ by \cite[Lemma~2]{lu89}. Hence, we consider just the case that $(L:M)\subseteq (P:M)$ and $(N:M)\subseteq (P:M)$. Now, we are going to show that if $N\not\subseteq P$, then $L \subseteq P$. For that, choose $x \in N \backslash P$. So, $x\not\in L$. If $(L:x)=(0)$, then $x+L \not\in S_{(0)}(\textbf{0}_{M/L})$, so $S_{(0)}(\textbf{0}_{M/L})\neq M/L$. Since $M$ is weak multiplication, it follows that $M/L$ is also a weak multiplication module. But every weak multiplication module over an integral domain is either torsion or torsion-free (see \cite[Proposition 3]{az03}). Hence $M/L$ is a torsion-free $R$-module.

    On the other hand, we have $(L:M) \subseteq (L:x)=(0)$. Thus $L \in \Spec_{(0)}M$ by \cite[Theorem 1]{lu84}. Therefore $L=(0)M=(\textbf{0}) \subseteq P$ as desired. Now let $(L:x)\neq (0)$ and $L=\bigcap_{\lambda \in \Lambda}P_{\lambda}$, where $P_{\lambda}$ are $\p_\lambda$-prime submodules of $M$ for each $\lambda \in \Lambda$. By assumption $P_{\lambda}=\p_{\lambda}M$. This implies that $$(L:x)=(\bigcap_{\lambda\in \Lambda}\p_{\lambda}M:x)=\bigcap_{\lambda\in \Lambda}(\p_{\lambda}M:x).$$ Suppose that $\Lambda'$ be a subset of $\Lambda$ such that for each $\lambda\in\Lambda'$, $x\not\in \p_{\lambda}M$. Since $x\not\in L$, hence $\Lambda'\neq\emptyset$. Now by \cite[Lemma~2.12]{mms02} and since $\dim(R)=1$, $$(0)\neq(L:x)=\bigcap_{\lambda\in \Lambda'}(\p_{\lambda}M:x)=\bigcap_{\lambda\in \Lambda'}\p_{\lambda}\subseteq(P:M).$$ Therefore, $(L:x)$ is a nonzero ideal of $R$, and so it is contained in only finitely many prime ideal by \cite[Proposition 9.1]{ati69}. Thus,  $\Lambda'$ is a finite set. It follows that there exists $\q \in\Lambda'$ such that $\q\subseteq \p$. This yields $L\subseteq \p M=P$ as desired.

\item  If $S_{(0)}(\textbf{0})=M$, then $\rad(\textbf{0})=M$, i.e., $\Spec(M)=\emptyset$, and so we are done. Therefore, we assume that  $S_{(0)}(\textbf{0})\neq M$. In this case $S_{(0)}(\textbf{0})$ is a $(0)$-prime submodule of $M$ by \cite[Lemma 4.5]{lu03}. We are going to show that every prime submodule of $M$ is extraordinary. Let $P$ be a prime submodule of $M$ and let $N$ and $L$ be two nonzero semiprime submodules of $M$ such that $N\cap L \subseteq P$. In view of above arguments  we take $x\in N\setminus P$. If $(L:x)=(0)$, then $(L:M)=(0)$ and by \cite[Result 1]{lu03}, $$S_{(0)}(\textbf{0}_{M/L})=S_{(0)}(\textbf{0})/L\subseteq \rad(\textbf{0})/L=(\textbf{0}).$$ Therefore, $M/L$ is a torsion-free $R$-module and $L$ is a $(0)$-prime submodule of $M$ by \cite[Theorem 1]{lu84}. By assumption of this part, $L=S_{(0)}(\textbf{0})\subseteq \rad(\textbf{0})\subseteq P$. Let $(L:x)\neq(0)$ and let $\{P_{\lambda}\}_{\lambda\in \Lambda}$ be a collection of $\p_{\lambda}$-prime submodules of $M$ such that $L=\bigcap_{\lambda\in \Lambda}P_{\lambda}$. If $\p_k=(0)$ for some $k\in \Lambda$, then $(P_k:M)=(S_{(0)}(\textbf{0}):M)=(0)$. Hence, $L\subseteq P_k=S_{(0)}(\textbf{0})\subseteq \rad(\textbf{0})\subseteq P$. Therefore, we may assume that $\p_\lambda\neq (0)$ for each $\lambda \in \Lambda$.
Since $\dim(R)=1$, we have $\p_{\lambda}=(\p_{\lambda}M:M)=(P_{\lambda}:M)$. Therefore, $\p_{\lambda}M$ is a $\p_{\lambda}$-prime submodule of $M$ by \cite[Proposition 2]{lu84}. By assumption of this part, $P_{\lambda}=\p_{\lambda}M$. This implies that $$(L:x)=(\bigcap_{\lambda\in \Lambda}\p_{\lambda}M:x)=\bigcap_{\lambda\in \Lambda}(\p_{\lambda}M:x).$$ Suppose that $\Lambda'$ be a subset of $\Lambda$ such that for each $\lambda\in\Lambda'$, $x\not\in \p_{\lambda}M$. Since $x\not\in L$, hence $\Lambda'\neq\emptyset$. Now, from the \cite[Lemma~2.12]{mms02}, we have, $$(0)\neq(L:x)=\bigcap_{\lambda\in \Lambda'}(\p_{\lambda}M:x)=\bigcap_{\lambda\in \Lambda'}\p_{\lambda}\subseteq(P:M).$$ By \cite[Proposition 9.1]{ati69}, $(L:x)$ is contained in finitely many prime ideal, i.e., $\Lambda'$ is finite. So, there exists some  $\lambda\in \Lambda'$ such that $\p_{\lambda}\subseteq (P:M)$. Therefore, $L\subseteq P$.
\end{enumerate}
\end{proof}

The next example shows that Part~(1) of Theorem \ref{maintop} is different from Part~(2).

\begin{eg}
Consider the $\mathbb{Z}$-module $M=\mathbb{Z}(p^{\infty})\oplus \mathbb{Z}$. It is easy to see that for every prime ideal $\p\in \Spec(\mathbb{Z})$, $|\Spec_{\p}(M)|\leq 1$ and $S_{(0)}(\textbf{0})= \rad(\textbf{0})$. By Theorem~\ref{maintop}, $M$ is a top module. We note that M is not weak
multiplication.
\end{eg}

\section{SOME TOPOLOGICAL PROPERTIES OF $ q\Spec(M)$}

Let $M$ be an $R$-module. Then for submodules $N$, $L$ and $N_i$ of $M$ we have
\begin{enumerate}
  \item $D(\textbf{0})= q\Spec(M)$ and $D(M) =\emptyset$,
  \item $\bigcap_{i\in I}D (N_i) = D (\sum_{i\in I}(N_i : M)M)$,
  \item $D (N)\cup D (L) = D (N \cap L)$.
\end{enumerate}
Now, we put $$\zeta(M)=\{ \, D(N) \mid N\leq M \, \}.$$
From (1), (2)  and (3) above, it is evident that for any module $M$ there exists a topology, $\tau$ say,
on $ q\Spec(M)$ having $\zeta(M)$ as the family of all closed sets. The topology $\tau$ is called the \emph{developed Zariski
topology on} $ q\Spec(M)$. For the reminder of this paper, for every ideal $I\in D(\Ann(M))$, $\overline{R}$ and $\overline{I}$
will denote respectively $R/\Ann(M)$ and $I/\Ann(M)$. Let $Y$ be a subset of $ q\Spec(M)$ for an $R$-module $M$.
We will denote the intersection of all elements in $Y$ by $\Im(Y )$ and the closure of $Y$ in $ q\Spec(M)$ with respect to the developed Zariski topology by $Cl(Y)$. The proof of next lemma is easy.

\begin{lem}\label{lemtop}
Let $I$ be a proper ideal of $R$ and $M$ be an $R$-module with submodules $N$ and $L$. Then we have
\begin{enumerate}
  \item If $(N : M) = (L : M)$, then $D(N)=D(L)$. The converse is also true if both $N$ and $L$ are quasi-prime submodules of $M$;
  \item $D(N)=\bigcup_{I\in D^R(N:M)} q\Spec_I(M)$;
  \item $D(N)=D((N:M)M)=\Omega^M((N:M)M)$;
  \item Let $Y$ be a subset of $q\Spec(M)$. Then $Y\subseteq D(N)$ if and only if  $(N:M)\subseteq (\Im(Y):M)$.
\end{enumerate}
\end{lem}

\begin{prop}\label{cont}\label{opcl}
Let $M$ be an $R$-module and $\psi :  q\Spec(M)\rightarrow q\Spec(R/\Ann(M))$ be the natural map.
\begin{enumerate}
  \item The natural map $\psi$ is continuous with respect to the developed Zariski topology.
  \item If $M$ is quasi-primeful, then $\psi$ is both closed and open; more precisely, for every submodule $N$ of $M$, $ \psi(D^M(N)) =D^{\overline{R}}(\overline{(N : M)})$ and $ \psi( q\Spec(M)-D^M(N)) = q\Spec(\bar{R})-D^{\overline{R}}(\overline{(N : M)})$.
  \item $\psi$ is bijective if and only if it is a homeomorphism.
\end{enumerate}
\end{prop}
\begin{proof}
(1) Let $I$ be an ideal of $R$ containing $\Ann(M)$ and let $L\in \psi^{-1}(D^{\overline{R}}(\bar{I}))$. There exists some $\bar{J}\in D^{\overline{R}}(\bar{I})$ such that $ \psi(L)=\bar{J}$. Hence, $J=(L:M)\supseteq I$ and $L\in D^M (IM)$. Now, let $K\in D^M (IM)$. Then $(K:M)\supseteq(IM:M)\supseteq I$, and so $K\in \psi^{-1}(D^{\overline{R}}(\bar{I}))$. Consequently, $\psi^{-1}(D^{\overline{R}}(\bar{I})) = D^M(IM)$, i.e., $\psi$ is continuous. (2) By part~(1), $\psi$ is a continuous map such that $\psi^{-1}(D^{\overline{R}}(\bar{I})) = D^M (IM)$ for every ideal $I$ of $R$ containing $\Ann(M)$. Hence, for every submodule $N$ of $M$, $\psi^{-1}(D^{\overline{R}}(\overline{(N : M)})) = D^M((N : M)M) = D^M (N)$. Since the natural map $\psi$ is surjective, $ \psi(D^M(N)) =\psi o \psi^{-1}(D ^{\overline{R}}(\overline{(N : M)})) = D^{\overline{R}}(\overline{(N : M)})$. Similarly, $ \psi( q\Spec(M)-D^M(N)) = q\Spec(\bar{R})-D^{\overline{R}}(\overline{(N : M)})$. (3) This follows from (1) and (2).
\end{proof}

\begin{thm}
Let $M$ be a quasi-primeful $R$-module. Then the following statements are equivalent:
\begin{enumerate}
  \item $ q\Spec(M)$ is connected;
  \item $ q\Spec(\bar{R})$ is connected;
  \item The ring $\bar{R}$ contains no idempotent other than $\bar{0}$ and $\bar{1}$.
\end{enumerate}
\noindent Consequently, if $R$ is a quasi-local ring, then both $q\Spec(M)$ and $ q\Spec(\bar{R})$ are connected.
\end{thm}
\begin{proof}
$(1) \Rightarrow (2)$ follows from that $\psi$ is a surjective and continuous.

For $(2) \Rightarrow (1)$, we assume that $ q\Spec(\bar{R})$ is connected. If $ q\Spec(M)$ is disconnected, then $ q\Spec(M)$ must contain a non-empty proper subset $Y$ that is both open and closed. Accordingly, $ \psi(Y)$ is a non-empty subset of $ q\Spec(\bar{R})$ that is both open and closed by Proposition \ref{opcl}. To complete the proof, it suffices to show that $ \psi(Y)$ is a proper subset of $ q\Spec(\bar{R})$ so that $ q\Spec(\bar{R})$ will be disconnected.

Since $Y$ is open, $Y =  q\Spec(M)- D^M(N)$ for some $N\leq M$ whence $ \psi(Y) = q\Spec(\bar{R})-D^{\overline{R}}(\overline{(N : M)})$ by Proposition \ref{opcl} again. Therefore, if $ \psi(Y) =  q\Spec(\bar{R})$, then $D^{\overline{R}}(\overline{(N : M)})=\emptyset$, and so $\overline{(N : M)}=\bar{R}$, i.e., $N = M$. It follows that $Y =  q\Spec(M)-D^M(N) = q\Spec(M)-D^M(M)=  q\Spec(M)$ which is impossible. Thus $ \psi(Y)$ is a proper subset of $ q\Spec(\bar{R})$.

For $(2) \Leftrightarrow (3)$, it is enough for us to show that $q\Spec(R)$ is disconnected if and only if $R$ has an idempotent $e \neq 0 ,1$. Suppose that $e \neq 0 ,1$ is an idempotent in $R$. Hence $R=Re\oplus R(1-e)$. It follows that $q\Spec(R)=(q\Spec(R)\setminus D^R(Re))\cup (q\Spec(R)\setminus D^R(R(1-e)))$ and $\emptyset=(q\Spec(R)\setminus D^R(Re))\cap (q\Spec(R)\setminus D^R(R(1-e)))$. This implies that $q\Spec(R)$ is disconnected. Now, we assume that $q\Spec(R)$ is disconnected. Thus $q\Spec(R)=D^R(I)\cup D^R(J)$ where $I$ and $J$ are two ideals of $R$. We have that $q\Spec(R)=D^R(I\cap J)$ and so, $I\cap J\subseteq \Im(q\Spec(R))$. Also, $\emptyset=D^R(I)\cap D^R(J)=D^R(I+ J)$. This implies that $I+J=R$. There exist $a\in I$ and $b\in J$ such that $a+b=1$. On the other hand, $$ab\in IJ\subseteq I\cap J\subseteq \Im(q\Spec(R))\subseteq\sqrt{(0)}.$$ So, $(ab)^n=0$ for some $n\in \mathbb{N}$. We have $1=(a+b)^n=a^n+b^n+abx$ where $x\in R$. Since $abx\in \sqrt{(0)}\subseteq \Rad(R)$, $a^n+b^n$ is a unit in $R$. Let $u$ be the inverse of $a^n+b^n$. Note that $ua^nb^n=0$. Thus $$ua^n=ua^n(u(a^n+b^n))=u^2a^{2n}+u^2a^nb^n=(ua^n)^2.$$ Similarly, $ub^n=(ub^n)^2$. If $ua^n=0$, then $a^n=0$, and so $1=b(b^{n-1}+ax)\in J$ which is contradiction because $D^R(J)\neq \emptyset$. Consequently, $ua^n$ and $ub^n$ are nonzero. On the other hand, if $ua^n=ub^n=1$, then $1=u(a^n+b^n)=ua^n+ub^n=1+1$, which is contradiction. We conclude that either $ua^n$ or $ub^n$ is idempotent.
\end{proof}

\begin{prop}\label{Cl}\label{Cl2}\label{t1}
Let $M$ be an $R$-module, $Y\subseteq  q\Spec(M)$ and let $L\in q\Spec_I(M)$.
\begin{enumerate}
  \item $D(\Im(Y))=Cl(Y)$. In particular $Cl(\{L\})=D (L)$;
  \item Let $M$ be a semiprimitive (resp. reduced) $R$-module and $\Max(M)$ (resp. $\Spec(M)$) be a non-empty connected subspace of $q\Spec(M)$. Then $q\Spec(M)$ is connected;
  \item If $(\textbf{0})\in Y$, Then $Y$ is dense in $q\Spec(M)$.
  \item The set $\{L\}$ is closed in $ q\Spec(M)$ if and only if
  \begin{enumerate}
    \item $I$ is a maximal element in $\{(N:M) | N\in  q\Spec(M)\}$, and
    \item $q\Spec_I(M)=\{L\}$.
    \end{enumerate}
  \item If $\{L\}$ is closed in $q\Spec(M)$, then $L$ is a maximal element of $q\Spec(M)$ and $q\Spec_I(M)=\{L\}$.
  \item $M$ is quasi-prime-embedding if and only if $q\Spec(M)$ is a $T_0$-space.
    \item $q\Spec(M)$ is a $T_1$-space if and only if $ q\Spec(M)$ is a $T_0$-space and for every element $L\in q\Spec(M)$, $(L:M)$ is a maximal element in $\{(N:M)~|~N\in q\Spec(M)\}.$
  \item If $q\Spec(M)$ is a $T_1$-space, then $q\Spec(M)$ is a $T_0$-space and every quasi-prime submodule is a maximal element of $q\Spec(M)$. The converse is also true, when $M$ is finitely generated.
  \item Let $(\textbf{0})\in q\Spec(M)$. Then $q\Spec(M)$ is a $T_1$-space if and only if $(\textbf{0})$ is the only quasi-prime submodule of $M$.
  \end{enumerate}
\end{prop}
\begin{proof}
\begin{enumerate}
  \item Clearly, $Y\subseteq D(\Im(Y ))$. Next, let $D(N)$ be any closed subset of $q\Spec(M)$
containing $Y$. Then $(L : M) \supseteq (N : M)$ for every $L\in Y$ so that $(\Im(Y ) : M)\supseteq (N : M)$. Hence, for every $Q\in D(\Im(Y ))$, $(Q : M)\supseteq (\Im(Y ) : M) \supseteq (N : M)$, namely $D (\Im(Y )) \subseteq D (N)$. This proves that $D (\Im(Y ))$ is the smallest closed subset of $ q\Spec(M)$ containing $Y$, hence $D (\Im(Y )) =Cl(Y )$.

  \item Let $M$ be reduced. Then by (1), we have $Cl(\Spec(M))=D(\Im(\Spec(M)))=D(\textbf{0})=q\Spec(M)$. Therefore, $q\Spec(M)$ is connected by \cite[p.150, Theorem 23.4]{mk99}. A similar proof is true for semiprimitive modules.

  \item is clear by (1).

  \item Suppose that $\{L\}$ is closed. Then $\{L\}=D(L)$ by (1). Let $N\in q\Spec(M)$ such that $(L:M)\subseteq (N:M)$. Hence, $N\in D(L)= \{L\}$, and so $q\Spec_I(M)=\{L\}$, where $I=(L:M)$. On the other hand we assume that (a) and (b) hold. Let $N\in Cl(\{L\})$. Hence, $(N:M)\supseteq (L:M)$ by~(1). By (a), $(N:M)=(L:M)$. So, $L=N$ by (b). This yields  $Cl(\{L\})=\{L\}$.

  \item Let $P\in q\Spec(M)$ such that $L\subseteq P$. Then $(L:M)\subseteq (P:M)$. i.e., $P\in D(L)=Cl(\{L\})=\{L\}$. Hence, $P=L$, and so $L$ is a maximal element of $q\Spec(M)$.

  \item We recall that a topological space is  $T_0$ if and only if the closures of distinct points are distinct. Now, the result follows from part~(1) and Proposition~\ref{t0inj}.

  \item We recall that a topological space is  $T_1$ if and only if every singleton subset is closed. The result follows from (4), (5) and (6).

  \item Trivially, $q\Spec(M)$ is a $T_0$-space and every it's singleton subset is closed. Every quasi-prime submodule is a maximal element of $q\Spec(M)$ by (5). Now, we suppose that $M$ is finitely generated. Thus, every quasi-prime submodule is maximal. Let $N\in q\Spec(M)$ such that $N\in Cl(\{L\})=D(L)$. Since $L$ is maximal, $(L:M)=(N:M)$. By Proposition~\ref{t0inj}, $N=L$. Hence, every singleton subset of $q\Spec(M)$ is closed. So, $q\Spec(M)$ is a $T_1$-space.

  \item Use part (8).
    \end{enumerate}
\end{proof}

\begin{eg}
Consider the $\mathbb{Z}$-module $M=\bigoplus_{p} \mathbb{Z}/{p\mathbb{Z}}$, where $p$ runs through the set of all prime integers. We will show that $q\Spec(M)$ is not a $T_1$-space. Note that $(\textbf{0}:M)=\Ann(M)=(0)$. Hence, $(\textbf{0})\in q\Spec(M)$. On the other hand, for each quasi-prime ideal $I$ of $\mathbb{Z}$, we have $(IM:M)=\sqrt{I}\in q\Spec(\mathbb{Z})$. So, $q\Spec(M)$ is infinite  and  $q\Spec(M)$ is not a $T_1$-space by Proposition~\ref{Cl}.
\end{eg}

\begin{rem}\label{remmax}
Let $M$ be a finitely generated (or co-semisimple) $R$-module. Since every quasi-prime submodule is contained in a maximal submodule, $q\Spec(M)$ is a $T_1$-space if and only if $q\Spec(M)$ is a $T_0$-space and $q\Spec(M)=\Max(M)$. Since $q\Spec(R)$ is always a $T_0$-space (see \cite[Theorem~4.1]{az08}), we have $q\Spec(R)$ is a $T_1$-space if and only if  $q\Spec(R)=\Max(R)$. If $R$ is absolutely flat, then by \cite[Theorem~2.1]{az08}, $q\Spec(R)=\Spec(R)=\Max(R)$. Therefore, $q\Spec(R)$ is a $T_1$-space. It is clear that if $M$ is free, then $q\Spec(M)$ is a $T_1$-space if and only if $M$ is isomorphic to $R$ and $q\Spec(R)$ is a $T_1$-space.
\end{rem}

\begin{thm}\label{t1multi}
Let $M$ be a finitely generated $R$-module. The following statements are equivalent
\begin{enumerate}
  \item $q\Spec(M)$ is a $T_1$-space.
  \item $M$ is a multiplication module and $q\Spec(M)=\Max(M)$.
\end{enumerate}
\end{thm}
\begin{proof}
Use Corollary~\ref{maincor}, Remark~\ref{remmax} and Proposition \ref{Cl}(6).
\end{proof}

\begin{cor}
Let $M$ be an $R$-module.
\begin{enumerate}
  \item Let $R$ be an integral domain. If $q\Spec(R)$ is a $T_1$-space, then $R$ is a field.
  \item If $M$ is Noetherian and $q\Spec(M)$ is a $T_1$-space, then $M$ is Artinian cyclic.
\end{enumerate}
\end{cor}
\begin{proof}
(1) By Remark~\ref{remmax}, we have $q\Spec(R)=\Max(R)$. But $(0)\in q\Spec(R)$ by assumption. Hence, $R$ is a field. (2) By Theorem~\ref{t1multi}, $M$ is multiplication and every prime submodule of $M$ is maximal. By \cite[Theorem 4.9]{beh06}, $M$ is Artinian. The result follows from \cite[Corollary 2.9]{zb88}.
\end{proof}

A topological space $X$ is said to be \emph{irreducible} if $X \neq \emptyset$ and if every pair of non-empty open sets in $X$ intersect, or equivalently if every non-empty open set is dense in $X$. A topological space $X$ is irreducible if for any decomposition $X=A_1\cup A_2$ with closed subsets $A_i$ of $X$ with $i = 1, 2$, we have $A_1 = X$ or $A_2 = X$. A subset $Y$ of $X$ is irreducible if it is irreducible as a subspace of $X$. An irreducible component of a topological space $A$ is a maximal irreducible subset of $X$.

 Both of  a singleton subset and its closure  in  $q\Spec(M)$ are irreducible. Now, applying (1) of Proposition \ref{Cl2}, we obtain that

\begin{cor}\label{cor4}
$D(L)$ is an irreducible closed subset of $ q\Spec(M)$ for every quasi-prime submodule $L$ of $M$.
\end{cor}

\begin{thm}\label{irrsub}
Let $M$ be an $R$-module and $Y \subseteq  q\Spec(M)$. Then $\Im(Y)$ is a quasi-prime submodule of $M$ if and only if $Y$ is an irreducible space.
\end{thm}
\begin{proof}
Let $\Im(Y)$ be a quasi-prime submodule of $M$. Let $Y\subseteq Y_1\cup Y_2$ where $Y_1$ and $Y_2$ are two closed subsets of $X$. Then there are submodules $N$ and $L$ of $M$ such that $Y_1=D(N)$ and $Y_2=D(L)$. Hence, $Y\subseteq D(N)\cup D(L)=D(N\cap L)$. By Lemma~\ref{lemtop}, $((N\cap L):M)\subseteq (\Im(Y):M)$. Since $(\Im(Y):M)$ is a quasi-prime ideal, either $(N:M)\subseteq (\Im(Y):M)$ or $(L:M)\subseteq (\Im(Y):M)$. By Lemma~\ref{lemtop}, either $Y\subseteq D(N)=Y_1$ or $Y\subseteq D(L)=Y_2$. This yields  $Y$ is irreducible.

Assume that $Y$ is an irreducible space. Let $I$ and $J$ be two ideals of $R$ such that $I\cap J \subseteq (\Im(Y):M)$. Suppose for contradiction that $I\not \subseteq (\Im(Y):M)$ and $J\not\subseteq (\Im(Y):M)$. Then $(IM:M)\not \subseteq (\Im(Y):M)$ and $(JM:M)\not \subseteq (\Im(Y):M)$. By Lemma~\ref{lemtop}, $Y\not \subseteq D(IM)$, $Y\not \subseteq D(JM)$. Let $P\in Y$. Then $(P:M)\supseteq (\Im(Y):M)\supseteq I\cap J$. This means that either $IM\subseteq (P:M)M$ or $JM\subseteq (P:M)M$. So, by Lemma~\ref{lemtop}, either $D(P)\subseteq D(IM)$ or $D(P)\subseteq D(JM)$. Therefore, $Y\subseteq D(IM)\cup D(JM)$ which is a contradiction to irreducibility of $Y$.
\end{proof}

\begin{eg}
Consider $M=\mathbb{Z}/p\mathbb{Z}\oplus \mathbb{Z}$ as a $\mathbb{Z}$-module, where $p$ is a prime integer. It is easy to see that $L=\mathbb{Z}/p\mathbb{Z}\oplus (0)$ and $N=(\bar{0})\oplus p\mathbb{Z}$ are prime submodules of $M$. We have $\Im(q\Spec(M))\subseteq L\cap N=(\textbf{0})$. Hence, $(\Im(q\Spec(M)):M)=((0):M)=(0)$ is a quasi-prime ideal of $\mathbb{Z}$. This implies that $\Im(q\Spec(M))$ is a quasi-prime submodule of $M$. By Theorem \ref{irrsub}, $q\Spec(M)$ is an irreducible space.
\end{eg}

\begin{cor}
Let $M$ be an $R$-module and $N\leq M$.
\begin{enumerate}
    \item $V^M(N)$ is irreducible if and only if $\rad(N)$ is a quasi-prime submodule.\label{p3}
    \item If $N$ is a $\p$-primary submodule of $M$ where $\p\in \Max(R)$, then $V^M(N)$ is irreducible.
    \item Let $R$ be a quasi-local ring. Then $\Max(M)$ is irreducible.
    \item The quasi-prime spectrum of every faithful reduced module over an integral domain is irreducible.
\end{enumerate}
\end{cor}
\begin{proof}
(1) Since $\rad(N)=\Im(V^M(N))$, result follows immediately from Theorem~\ref{irrsub}. (2) Use part~(\ref{p3}) and \cite[Corollary 5.7]{lu03}. (3) Let $\m$ be the unique maximal ideal of $R$. By \cite[p.63, Proposition 4]{lu84}, $(H:M)=\m$ for each $H\in \Max(M)$. By Lemma~\ref{lemaux}(2), $\bigcap_{H\in \Max(M)} H=\Im(\Max(M))$ is a quasi-prime submodule.  By Theorem \ref{irrsub}, $\Max(M)$  is irreducible. (4) Since $M$ is reduced, $(\Im(q\Spec(M)):M)\subseteq(\Im(\Spec(M)):M)=(\bigcap_{P\in \Spec(M)} P:M)=((\textbf{0}):M)=(0)\in \Spec(R)$. The result follows from  Theorem \ref{irrsub}.
\end{proof}

\begin{eg}\label{egs}
\begin{enumerate}
  \item Let $M=\mathbb{Z} \oplus \mathbb{Z}(p^{\infty})$ be a $\mathbb{Z}$-module. Then by Theorem~\ref{irrsub}, $\Spec(M)$ is an irreducible space because $\Im(\Spec(M))=(0) \oplus \mathbb{Z}(p^{\infty})$ is a prime submodule of $M$.
  \item Let $M=\mathbb{Q} \oplus \mathbb{Z}/p\mathbb{Z}$ be a $\mathbb{Z}$-module. By Theorem \ref{irrsub}, $\Max(M)$ is an irreducible subset of $ q\Spec(M)$ because $\Rad(M)=\mathbb{Q} \oplus (0)$.
\end{enumerate}
\end{eg}

\begin{cor}\label{zero}
Let $M$ be an $R$-module such that $(\textbf{0})\in q\Spec(M)$. Then $q\Spec(M)$ is an irreducible space. In particular, if $R$ is an integral domain and $M$ is a torsion-free $R$-module, then $q\Spec(M)$ is an irreducible space. Moreover, $q\Spec(R)$ is an irreducible space, if $R$ is an integral domain.
\end{cor}
\begin{proof}
Use Theorem~\ref{irrsub} and \cite[Lemma~4.5]{lu03}.
\end{proof}

\begin{eg}
Consider the faithful $\mathbb{Z}$-module $M=\bigoplus_{p} \mathbb{Z}/{p\mathbb{Z}}$, where $p$ runs through the set of all prime integers. Then by Corollary \ref{zero}, $\Spec(M)$ is an irreducible space.
\end{eg}

Let $Y$ be a closed subset of a topological space. An element $y\in Y$ is called a \emph{generic point} of $Y$ if $Y=Cl(\{y\})$. In Proposition \ref{Cl2} (1), we have seen that every element $L$ of $ q\Spec(M)$ is a generic point of the irreducible closed subset $D(L)$ of $ q\Spec(M)$. Note that a generic point of a closed subset $Y$ of a topological space is unique if the topological space is a $T_{0}$-space.

\begin{thm}\label{irclsub}\label{cores}\label{irrmin}
Let $M$ be an $R$-module and $Y \subseteq q\Spec(M)$.
\begin{enumerate}
  \item  Then $Y$ is an irreducible closed subset of $ q\Spec(M)$ if and only if $Y = D^M(L)$ for some $L \in  q\Spec(M)$. Thus every irreducible closed subset of $ q\Spec(M)$ has a generic point.
  \item If $M$ is quasi-prime-embedding, then the correspondence $D^M(L)\mapsto L$ is a bijection of the set of irreducible components of $ q\Spec(M)$ onto the set of minimal elements of $ q\Spec(M)$ with respect to inclusion.
  \item Let $M$ be a quasi-primeful $R$-module. Then the set of all irreducible components of $ q\Spec(M)$ is of the form $$T=\{ D^M(IM) \mid \text{$I$ is a minimal element of $D^R(\Ann(M))$ w.r.t inclusion} \}.$$
  \item Let $R$ be an arithmetical Laskerian ring and $M$ be a nonzero quasi-primeful $R$-module. Then $q\Spec(M)$ has finitely many irreducible components.
\end{enumerate}
\end{thm}
\begin{proof}
\begin{enumerate}
  \item It is clear $Y=D(L)$ is an irreducible closed subset of $ q\Spec(M)$ for any $L\in  q\Spec(M)$ by Corollary~\ref{cor4}. Conversely, if $Y$ is an irreducible closed subset of $ q\Spec(M)$, then $Y=D(N)$ for some $ N\leq M$ and $L:=\Im (Y)=\Im(D(N))\in  q\Spec(M)$ by Theorem~\ref{irrsub}. Hence, $Y=D(N)=D(\Im(D(N)))=D(L)$ as desired.

  \item Let $Y$ be an irreducible component of $ q\Spec(M)$. Since each irreducible component of $ q\Spec(M)$ is a maximal element of the set $\{D(N) \mid N\in  q\Spec(M)\}$ by (1), we have $Y=D(L)$ for some $L\in  q\Spec(M)$. Obviously, $L$ is a minimal element of $ q\Spec(M)$, for if $T\in q\Spec(M)$ with $T\subseteq L$, then $D(L)\subseteq D(T)$. So $L=T$ due to the maximality of $D(L)$ and Proposition~\ref{t0inj}. Let $L$ be a minimal element of $ q\Spec(M)$ with $D(L)\subseteq D(N)$ for some $N\in  q\Spec(M)$. Then $L\in D(N)$ whence $(N:M)M\subseteq L$. By Lemma~\ref{lemaux}, $(N:M)M$ belongs to $q\Spec(M)$. Hence, $L=(N:M)M$ due to the minimality of $L$. By Lemma~\ref{lemtop}, $D(N)=D((N:M)M)=D(L)$. This implies that $D(L)$ is an irreducible component of $ q\Spec(M)$, as desired.

  \item Let $Y$ be an irreducible component of $ q\Spec(M)$. By part~(1), $Y=D^M(L)$ for some $L\in  q\Spec(M)$. Hence, $Y=D^M(L)=D^M((L:M)M)$ by Lemma~\ref{lemtop}. So, we have $l:=(L:M)\in D^R(\Ann(M))$. We must show that $l$ is a minimal element of $D^R(\Ann(M))$ w.r.t inclusion. To see this let $q\in D^R(\Ann(M))$ and $q\subseteq l$. Then $q/\Ann(M)\in  q\Spec(R/\Ann(M))$, and there exists an element $Q\in q\Spec(M)$ such that $(Q:M)=q$ because $M$ is quasi-primeful. So, $Y=D^M(L)\subseteq D^M(Q)$. Hence, $Y=D^M(L)=D^M(Q)$ due to the maximality of $D^M(L)$. By Proposition~\ref{Cl2}, we have that $l=q$. Conversely, let $Y\in T$. Then there exists a minimal element $I$ in $D^R(\Ann(M))$ such that $Y=D^M(IM)$. Since $M$ is quasi-primeful, there exists an element $N\in  q\Spec(M)$ such that $(N:M)=I$. So, $Y=D^M(IM)=D^M((N:M)M)=D^M(N)$, and so $Y$ is irreducible by part~(1). Suppose that $Y=D^M(N)\subseteq D^M(Q)$, where $Q$ is an element of $ q\Spec(M)$. Since $N\in D^M(Q)$ and $I$ is minimal, it follows that $(N:M)=(Q:M)$. Now, by Lemma~\ref{lemtop}, $$Y=D^M(N)=D^M((N:M)M)=D^M((Q:M)M)=D^M(Q).$$

  \item By assumption, the set of quasi-prime ideals are exactly the set of primary ideals (see Remark \ref{6}). If $I$ is a minimal element of $D^R(\Ann(M))$ and $\Ann(M) = \cap_{i=1}^n Q_i$ is a minimal primary decomposition of $\Ann(M)$, then  $Q_i \subseteq I$  for some $1\leq i \leq n$(Since $I$ is quasi-prime and $\cap_{i=1}^n Q_i\subseteq I$). By minimality of $I$, we get $I = Q_i$. Therefore, irreducible components of $q\Spec(M)$ are the form $D^M(Q_iM)$, by part (3).
\end{enumerate}
\end{proof}

We introduce a base for the developed Zariski topology on $q\Spec(M)$ for any $R$-module $M$. For each $a\in R$, we define $\Gamma_M(a)= q\Spec(M)-D(aM)$. Then every $\Gamma_M(a)$ is an open set of $ q\Spec(M)$, $\Gamma_M(0) =\emptyset$, and $\Gamma_M(1)=q\Spec(M)$.

\begin{prop}\label{base}
For any $R$-module $M$, the set $B =\{ \Gamma_M(a)\mid a\in R\}$ forms a base
for the developed Zariski topology on $ q\Spec(M)$.
\end{prop}
\begin{proof}
We may assume that $ q\Spec(M)\neq\emptyset$. Let $U$ be any open subset in $ q\Spec(M)$. There exists a submodule $N$ of $M$ such that
\begin{eqnarray*}
  U &=& q\Spec(M)-D(N)=q\Spec(M)-D((N:M)M) \\
   &=& q\Spec(M)-D(\sum_{a_i\in(N:M)}a_iM) \\
   &=& q\Spec(M)-D(\sum_{a_i\in(N:M)}(a_iM:M)M) \\
   &=& q\Spec(M)-\bigcap_{a_i\in(N:M)} D(a_iM) \\
   &=& \bigcup_{a_i\in(N:M)} \Gamma_M(a_i).
\end{eqnarray*}
\end{proof}

\begin{prop}\label{XD}
Let $M$ be an $R$-module, $a\in R$ and $\psi :  q\Spec(M)\rightarrow q\Spec(R/\Ann(M))$ be the natural map of $ q\Spec(M)$.
\begin{enumerate}
  \item $\psi^{-1}(\Gamma_{\bar{R}}(\bar{a}))=\Gamma_M(a)$;
  \item $\psi(\Gamma_M(a))\subseteq \Gamma_{\bar{R}}(\bar{a})$. If $M$ is quasi-primeful, then $\psi(\Gamma_M(a))= \Gamma_{\bar{R}}(\bar{a})$;
  \item If $M$ is quasi-primeful, then $q\Spec(M)$ is a compact space.
  \item If $M$ is finitely generated multiplication, then $q\Spec(M)$ is compact.
\end{enumerate}
\end{prop}
\begin{proof}
\begin{enumerate}
  \item By Proposition \ref{cont}, we have
\begin{eqnarray*}
  \psi^{-1}(\Gamma_{\bar{R}}(\bar{a})) &=& \psi^{-1}(q\Spec(\bar{R})-D(\bar{a}\bar{R})) \\
    &=& q\Spec(M)- \psi^{-1}(D(\bar{a}\bar{R})) \\
    &=& q\Spec(M)-D(aM)=\Gamma_M(a).
\end{eqnarray*}
  \item This follows from (1).
  \item By Proposition~\ref{base}, the set $B =\{ \Gamma_M(a)\mid a\in R\}$ is a base for the developed Zariski topology on $q\Spec(M)$. For any open cover of $q\Spec(M)$, there is a family $\{a_{\lambda}\in R | \lambda \in \Lambda\}$ of elements of $R$ such that $q\Spec(M)=\bigcup_{\lambda\in\Lambda} \Gamma_M(a_\lambda)$ and for each $\lambda \in \Lambda$, there is an open set in the covering containing  $\Gamma_M(a_\lambda)$. By part (2),
      \begin{eqnarray*}
        q\Spec(\bar{R}) &=& \Gamma_{\bar{R}}(\bar{1}) \\
         &=& \psi(\Gamma_M(1)) \\
         &=& \psi(q\Spec(M))\subseteq \bigcup_{\lambda\in\Lambda} \psi(\Gamma_M(a_\lambda))\\
         &=& \bigcup_{\lambda\in\Lambda} \Gamma_{\bar{R}}(\bar{a}_\lambda).
      \end{eqnarray*}
      By \cite[Theorem 4.1]{az08}, $q\Spec(\bar{R})$ is compact, hence there exists a finite subset $\Lambda '$ of $\Lambda$ such that $q\Spec(\bar{R})\subseteq \bigcup_{\lambda\in\Lambda '} \Gamma_{\bar{R}}(\bar{a}_\lambda)$. By part (1),
      \begin{eqnarray*}
  q\Spec(M) &=& \Gamma_M(1) \\
    &=& \psi^{-1}(\Gamma_{\bar{R}}(\bar{1}))\\
    &=& \psi^{-1}(q\Spec(\bar{R}))\subseteq \bigcup_{\lambda\in\Lambda '} \psi^{-1}(\Gamma_{\bar{R}}(\bar{a}_\lambda))\\
    &=&\bigcup_{\lambda\in\Lambda '} \Gamma_M(a_\lambda).
\end{eqnarray*}
  \item Let $\{D^M(N_{\lambda})\}_{\lambda \in \Lambda}$ be an arbitrary family of closed subsets of $ q\Spec(M)$, where $N_{\lambda}\leq M$ for each $\lambda \in \Lambda$ such that $\bigcap_{\lambda \in \Lambda} D^M(N_{\lambda})=\emptyset$. Hence, we have $D^M(\sum_{\lambda\in \Lambda}(N_{\lambda} : M)M)=\emptyset$. Since $M$ is multiplication, $\Omega^M(\sum_{\lambda\in \Lambda}(N_{\lambda} : M)M)=\emptyset$, so $M=\sum_{\lambda\in \Lambda}(N_{\lambda} : M)M$. Since $M$ is finitely generated, there exists a finite subset $\Lambda '$ of $\Lambda$ such that $M=\sum_{\lambda\in \Lambda'}(N_{\lambda} : M)M$. This completes the proof.
\end{enumerate}
\end{proof}

A topological space $X$ is said to be \emph{Noetherian} if the open subsets of $X$ satisfy the ascending chain condition. Since closed subsets are complements of open subsets, it comes to the same thing to say that the closed subsets of $X$ satisfy the descending chain condition.

\begin{thm}\label{thm2.1}\label{accring}
Let $M$ ba an $R$-module.
\begin{enumerate}
  \item If $M$ satisfies $ACC$ on quasi-semiprime submodules, then $ q\Spec(M)$ is a Noetherian topological space. In particular, quasi-prime spectrum of every Noetherian module is a Noetherian topological space (see \cite[Theorem~4.2]{az08}).
  \item If for every submodule $N$ of $M$ there exists a finitely generated submodule $L$ of $N$ such that $\Im(\Omega^M(N))=\Im(\Omega^M(L))$, then $ q\Spec(M)$ is a Noetherian topological space.
  \item If $R$ satisfies $ACC$ on quasi-semiprime ideals, then $ q\Spec(M)$ is a Noetherian topological space. In particular, for every module $M$ over a Noetherian ring, $ q\Spec(M)$ is a Noetherian topological space.
\end{enumerate}
\end{thm}
\begin{proof}
\begin{enumerate}
  \item Let $D(N_{1})\supseteq D(N_{2})\supseteq \cdots $ be a descending chain of closed subsets of $ q\Spec(M)$. We have an ascending chain of quasi-semiprime submodules of $M$, $\Im(D(N_{1}))\subseteq \Im(D(N_{2}))\subseteq \cdots $ which is stationary by assumption.
 So, there exists a positive integer $k$ such that $\Im(D(N_{k}))= \Im(D(N_{k+i}))$, for each $i=1,2, \ldots$ \,. By Proposition \ref{Cl}, $D(N_{k})= D(N_{k+i})$, and so $ q\Spec(M)$ is a Noetherian topological space.
 \item Let $N_1\subseteq N_2\subseteq N_3\subseteq \cdots$ be an ascending chain of
quasi-semiprime submodules of $M$, and let $N=\cup_iN_i$. By assumption, there
exists a finitely generated submodule $L$ of $N$ such that $\bigcap_{P\in \Omega^M(N)}P =\bigcap_{Q\in \Omega^M(L)}Q$.
Hence there exists a positive integer $n$ such that $L\subseteq N_n$. Then
$$\bigcap_{P\in \Omega^M(N)}P =\bigcap_{Q\in \Omega^M(L)}Q\subseteq N_n \subseteq N \subseteq \bigcap_{P\in \Omega^M(N)}P,$$ so that $N_n = N_{n+1} = N_{n+2} =\cdots$. Hence, $M$ satisfies $ACC$ on quasi-semiprime submodules. By (1), $q\Spec(M)$ is a Noetherian topological space.
 \item Let $D(N_{1})\supseteq D(N_{2})\supseteq \cdots $ be a descending chain of closed subsets of $ q\Spec(M)$. By assumption, there exists a positive integer $k$ such that $(\Im(D(N_{k})):M)M= (\Im(D(N_{k+i})):M)M$, for each $i=1,2, \ldots$ \,. By Lemma~\ref{lemtop}, $D(\Im(D(N_{k})))= D(\Im(D(N_{k+i})))$. By Proposition~\ref{Cl}, $D(N_{k})= D(N_{k+i})$, and so $ q\Spec(M)$ is a Noetherian space.
\end{enumerate}
\end{proof}

\begin{rem}\label{remnoecom}
Let $X$ be a Noetherian topological space. Then every subspace of $X$ is compact. In particular, $X$ is compact (see \cite[p. 79, Ex. 5]{ati69}).
\end{rem}

As a consequence of Remark \ref{remnoecom}, we have

\begin{cor}\label{cor2.5}
For an $R$-module $M$, $ q\Spec(M)$ is a compact space in each of the following cases.
\begin{enumerate}
  \item $M$ satisfies $ACC$ on quasi-semiprime submodules;
  \item $R$ satisfies $ACC$ on quasi-semiprime ideals.
\end{enumerate}
\end{cor}

For example, quasi-prime spectrum of every $\mathbb{Z}$-module is compact space.

\begin{prop}\label{fin}\label{finimini}
Let $M$ be a quasi-prime-embedding\, $R$-module. If $ q\Spec(M)$ is a Noetherian  space, then
\begin{enumerate}
  \item Every ascending chain of quasi-prime submodules of $M$ is stationary;
  \item $ q\Spec(M)$ has finitely many minimal element. In particular, every multiplication module over a Noetherian ring has finitely many minimal quasi-prime submodules.
\end{enumerate}
\end{prop}
\begin{proof}
(1) Let $N_1\subseteq N_2 \subseteq \cdots$ be an ascending chain of quasi-prime submodules of $M$. Then $D(N_1)\supseteq D(N_2) \supseteq \cdots$ is a descending chain of closed subsets of $ q\Spec(M)$, which is stationary by assumption. There exists an integer $k\in \mathbb{N}$ such that $D(N_k)=D(N_{k+i})$ for each $i\in \mathbb{N}$. By Proposition~\ref{t0inj}, we have $N_k=N_{k+i}$ for each $i\in \mathbb{N}$. This completes the proof.
(2) Since every Noetherian topological space has finitely many irreducible components, the result follows from Theorem \ref{cores}(2). For last statement, use Corollary~\ref{maincor} and Theorem~\ref{accring}.
\end{proof}

We recall that if $X$ is a finite space, then $X$ is a $T_1$ if and only if $X$
is the discrete space. We also recall that a topological space is called  Hausdorff  if any two distinct points possess disjoint neighborhoods. So, we have the following corollary.

\begin{cor}
Let $R$ be a Noetherian ring and $M$ be a finitely generated $R$-module. Then the following statements are equivalent:
\begin{enumerate}
  \item $q\Spec(M)$ is a Hausdorff space;
  \item $q\Spec(M)$ is a $T_1$-space;
  \item $q\Spec(M)$ is a discrete space;
  \item $M$ is a multiplication module and $q\Spec(M)=\Max(M)$.
\end{enumerate}
\end{cor}
\begin{proof}
$(1)\Rightarrow(2)$ and $(3)\Rightarrow(1)$ are clear. $(2)\Leftrightarrow(4)$ follows form Theorem~\ref{t1multi}. $(2)\Rightarrow(3)$ By Proposition~\ref{finimini}, $M$ has finitely many minimal quasi-prime submodules. By Theorem~\ref{t1multi}, $q\Spec(M)$ is finite. Therefore, $q\Spec(M)$ is a discrete space.
\end{proof}

\section*{ACKNOWLEDGMENTS}
The authors  specially thank the referee for the useful suggestions and
comments.

\providecommand{\bysame}{\leavevmode\hbox to3em{\hrulefill}\thinspace}
\providecommand{\MR}{\relax\ifhmode\unskip\space\fi MR }
\providecommand{\MRhref}[2]{%
  \href{http://www.ams.org/mathscinet-getitem?mr=#1}{#2}
}
\providecommand{\href}[2]{#2}


\begin{thebibliography}{MMS98}

\bibitem[AF92]{ful92}
F.~W. Anderson and K.~R. Fuller, \emph{Rings and categories of modules},
  Springer-Verlag, New York, no.~13, Graduate Texts in Math., 1992.

\bibitem[AM69]{ati69}
M.F. Atiyah and I.G. Macdonald, \emph{Introduction to commutative algebra},
  Addison-Wesley, 1969.

\bibitem[AS95]{Abu95}
S.~Abu-Saymeh, \emph{On dimensions of finitely generated modules}, Comm.
  Algebra \textbf{23} (1995),  1131--1144.

\bibitem[Azi03]{az03}
A.~Azizi, \emph{Weak multiplication modules}, Czechoslovak Math. J. \textbf{53}
  (2003),  529--534.

\bibitem[Azi08]{az08}
\bysame, \emph{Strongly irreducible ideals}, J. Aust. Math. Soc. \textbf{84}
  (2008), 145--154.

\bibitem[Bar81]{Bar81}
A.~Barnard, \emph{Multiplication modules}, J. Algebra \textbf{71} (1981),
   174--178.

\bibitem[Beh06]{beh06}
M.~Behboodi, \emph{A generalization of the classical krull dimension for
  modules}, J. Algebra \textbf{305} (2006), 1128--1148.

\bibitem[BH08a]{behI08}
M.~Behboodi and M.~R. Haddadi, \emph{Classical \text{Zariski} topology of
  modules and spectral spaces \text{I}}, Int. Electron. J. Algebra \textbf{4}
  (2008), 104--130.

\bibitem[BH08b]{behII08}
\bysame, \emph{Classical \text{Zariski} topology of modules and spectral spaces
  \text{II}}, Int. Electron. J. Algebra \textbf{4} (2008), 131--148.

\bibitem[BKK04]{bkk04}
M.~Behboodi, O.~A.~S. Karamzadeh, and H.~Koohy, \emph{Modules whose certain
  submodules are prime}, Vietnam J. Math \textbf{32} (2004), 303--317.

\bibitem[Bou72]{Bou70}
N.~Bourbaki, \emph{Commutative algebra, chap. 1-7}, Paris: Hermann, 1972.

\bibitem[Dau78]{Dau78}
J.~Dauns, \emph{Prime modules}, J. Rein. Ang. Math. \textbf{298} (1978),
  165--181.

\bibitem[EBS88]{zb88}
Z.A. El-Bast and P.F. Smith, \emph{Multiplication modules}, Comm. Algebra
  \textbf{16} (1988),  755--779.

\bibitem[HRR02]{hen02}
W.~J. Heinzer, L.~J. Ratlif, and D.~E. Rush, \emph{Strongly irreducible ideals
  of a commutative ring}, J. Pure Appl. Algebra \textbf{166} (2002), 267--275.

\bibitem[Jen66]{Jen66}
C.~Jensen, \emph{Arithmetical rings}, Acta Mathematica Scientiarum Hungaricae
  \textbf{17} (1966), 115--123.

\bibitem[Lu84]{lu84}
Chin-Pi Lu, \emph{Prime submodules of modules}, Comment. Math. Univ. St. Pauli
  \textbf{33} (1984),  61--69.

\bibitem[Lu89]{lu89}
\bysame, \emph{M-radicals of submodules in modules}, Math. Japonica \textbf{34}
  (1989), 211--219.

\bibitem[Lu95]{lu95}
\bysame, \emph{Spectra of modules}, Comm. Algebra \textbf{23} (1995),
  3741--3752.

\bibitem[Lu99]{lu99}
\bysame, \emph{The \text{Zariski} topology on the prime spectrum of a module},
  Houston J. Math. \textbf{25} (1999),  417--432.

\bibitem[Lu03]{lu03}
\bysame, \emph{Saturations of submodules}, Comm. Algebra \textbf{31} (2003),
   2655 -- 2673.

\bibitem[Lu07]{lu07}
\bysame, \emph{A module whose prime spectrum has the surjective natural map},
  Houston J. Math. \textbf{33} (2007),  125--143.

\bibitem[MM92]{mm92}
R.L. McCasland and M.E. Moore, \emph{Prime submodules}, Comm. Algebra
  \textbf{20} (1992),  1803--1817.

\bibitem[MMS97]{mms97}
R.L. McCasland, M.E. Moore, and P.F. Smith, \emph{On the spectrum of a module
  over a commutative ring}, Comm. Algebra \textbf{25} (1997),  79--103.

\bibitem[MMS98]{mms98}
\bysame, \emph{An introduction to zariski spaces over zariski topologies},
  Rocky Mountain J. Math. \textbf{28} (1998),  1357--1369.

\bibitem[MS02]{mms02}
M.~E. Moore and Sally~J. Smith, \emph{Prime and radical submodules of modules
  over commutative rings}, Comm. Algebra \textbf{30} (2002),  5037 --
  5064.

\bibitem[Mun99]{mk99}
J.~R. Munkres, \emph{Topology}, second ed., Prentice Hall, New Jersey, 1999.

\bibitem[Nor68]{nor68}
D.~G. Northcott, \emph{Lesson on rings, modules and multiplicities}, Cambridge:
  Cambridge University Press, 1968.

\end{thebibliography}
\end{document}